\documentclass[a4paper]{article}

\usepackage[top=3.5cm,bottom=3.5cm,left=3.5cm,right=3.5cm]{geometry}

\usepackage{palatino}
\usepackage{latexsym}

\usepackage[mathscr]{eucal}
\usepackage{amsmath}
\usepackage{amsthm}
\usepackage{amsfonts}
\usepackage{amssymb}
\usepackage{amscd}
\usepackage{color}
\usepackage{graphicx}
\usepackage{graphics}
\usepackage{pifont}
\usepackage{subfigure}
\usepackage[makeroom]{cancel}
\usepackage[normalem]{ulem}
\usepackage[dvipsnames]{xcolor}
\usepackage{tikz}

\theoremstyle{plain}
\newtheorem{theorem}{Theorem}

\newtheorem{lemma}[theorem]{Lemma}

\newtheorem{definition}[theorem]{Definition}
\newtheorem{corollary}[theorem]{Corollary}

\title{Small-time bilinear control of Schr\"odinger equations\newline with application to rotating linear molecules}

\begin{document}
\author{Thomas Chambrion and Eugenio Pozzoli\footnote{The authors are with Institut de Mathématiques de Bourgogne,
 UMR 5584, CNRS, Université Bourgogne Franche-Comté, F-21000 Dijon, France. (thomas.chambrion@u-bourgogne.fr) (eugenio.pozzoli@u-bourgogne.fr)}}
\maketitle
\abstract{In \cite{duca-nersesyan} Duca and Nersesyan proved a small-time controllability property of nonlinear Schr\"odinger equations on a $d$-dimensional torus $\mathbb{T}^d$.
In this paper we study a similar property, in the linear setting, starting from a closed Riemannian manifold. We then focus on the 2-dimensional sphere $S^2$, which models the bilinear control of a rotating linear top: as a corollary, we obtain the approximate controllability in arbitrarily small times among particular eigenfunctions of the Laplacian of $S^2$.
}\\ 

\section{Introduction}
\subsection{The model}
Let $M$ be a smooth manifold equipped with a Riemannian metric $g$. In order to simplify the analysis, we require $M$ to be closed (i.e., boundaryless and compact). In this paper we deal with the controllability properties of the following bilinear Schr\"odinger equation
\begin{equation}\label{eq:schro}
i \frac{\partial}{\partial t}\psi(q,t)=\left(-\Delta_g+V(q)+\sum_{j=1}^m u_j(t)W_j(q)\right)\psi(q,t),
\end{equation}
where we assume that the initial datum $\psi(\cdot,t=0)=\psi_0(\cdot)$ belongs to the Hilbert space $L^2(M,\mathbb{C})$ of complex functions on $M$ that are square integrable w.r.t. the Riemannian volume $\omega_g$: i.e., $\psi\in L^2(M,\mathbb{C})$ if
$$\|\psi\|_{L^2(M,\mathbb{C})}:=\left(\int_{M}|\psi(q)|^2\omega_g(q)\right)^{1/2}<\infty. $$
In \eqref{eq:schro}, $\Delta_g={\rm div}_{\omega_g}\circ \nabla_g$ is the Laplace-Beltrami operator of $(M,g)$ and represents the kinetic energy, where ${\rm div}_{\omega_g}$ and $\nabla_g$ are respectively the divergence w.r.t. the Riemannian volume and the Riemannian gradient. Moreover, $V\in L^\infty(M,\mathbb{R})$ and $W_1,\dots,W_m\in C^\infty(M,\mathbb{R})$ are functions on $M$ (that we identify with multiplicative operators on $L^2(M,\mathbb{C})$) representing respectively a free potential energy and potentials of interaction that can be tuned by means of a time-dependent control law $(u_1(t),\dots,u_m(t))$. 

An example of system that we study in detail in this paper is given by the following Schr\"odinger equation on the two-dimensional sphere $M=S^2:=\{(x,y,z)\in\mathbb{R}^3\mid x^2+y^2+z^2=1\}$:
\begin{equation}\label{eq:S2}
i \frac{\partial}{\partial t}\psi(x,y,z,t)=\left(-\Delta_{S^2}+u_1(t)x+u_2(t)y+u_3(t)z\right)\psi(x,y,z,t),
\end{equation}
$\psi(\cdot,0)=\psi_0(\cdot)\in L^2(S^2,\mathbb{C})$. The expression of the Riemannian volume, the potentials of interaction and the Laplace-Beltrami operator of $S^2$ in spherical coordinates $(\alpha,\beta)\in[0,2\pi)\times [0,\pi]$ are given by
\begin{align*}
\omega_{S^2}&=\sin(\beta)d\alpha d\beta,\\
x&=\cos(\alpha)\sin(\beta),\quad y=\sin(\alpha)\sin(\beta),\quad z=\cos(\beta),\\
\Delta_{S^2}&=\frac{1}{\sin(\beta)}\frac{\partial}{\partial \beta}\left(\sin(\beta)\frac{\partial}{\partial \beta}\right)+\frac{1}{\sin^2(\beta)}\frac{\partial^2}{\partial \alpha^2}.
\end{align*}
System \eqref{eq:S2} is used in molecular physics to model the bilinear control in dipolar approximation of a rotating rigid linear 
molecule in the space by means of three orthogonal electric fields \cite{rabitz} (see Fig.\ref{fig:linear-top}). The capability of controlling molecular rotations has applications in physics ranging from chirality detection \cite{PattersonNature13} to quantum error correction \cite{victor}.

System \eqref{eq:S2} is known to be globally approximately controllable in large times \cite{BCS} (i.e., it is possible to steer any initial state to any neighborhood of any final state having the same norm by choosing suitable controls). Extensions of global approximate controllability to rigid symmetric and asymmetric molecules described by bilinear Schr\"odinger equations on the group of rotations $M=SO(3)$ have been obtained in \cite{Ugo-Mario-Io-symmetrictop,asymm-top}.
\begin{figure}\begin{center}
\includegraphics[width=0.35\linewidth, draft = false]{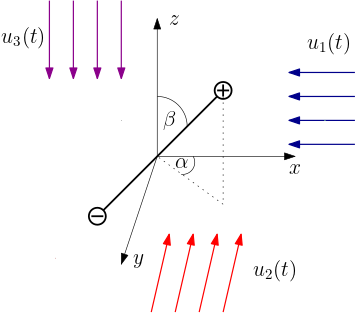}
\caption{Three orthogonal electric fields to control the rotation of a rigid linear molecule in $\mathbb{R}^3.$}\label{fig:linear-top}\end{center}\end{figure}
\subsection{Small-time approximate controllability}
The controllability properties of \eqref{eq:schro} have raised much interest across the mathematical community of partial differential equations in the last two decades (e.g., \cite{mirra,Coron,nerse,BCMS,laurent}), due to the relevance of such questions in physical applications such as spectroscopy and quantum information theory. System \eqref{eq:schro} is generically globally approximately controllable in large times \cite{MS-generic}. Here we focus on controllability properties holding in arbitrarily small times. This is an important subject because quantum systems undergo decoherence and relaxation effects, and the Schr\"odinger equation is an adequate physical model only for small times.

Being $M$ compact, the previously stated hypothesis on the potentials $V,W_1,\dots,W_m$ guarantee that they are bounded self-adjoint multiplicative operators on $L^2(M,\mathbb{R})$. Being $M$ boundaryless, the drift operator $-\Delta_g+V$ is self-adjoint on the domain $H^2(M,\mathbb{C})=\{\psi\in L^2(M,\mathbb{C})\mid \Delta_g \psi \in L^2(M,\mathbb{C}) \text{ weakly} \}$, and given any initial datum $\psi_0\in L^2(M,\mathbb{C})$ and any control $u\in L^1_{\rm loc}(\mathbb{R},\mathbb{R}^m)$ one can then define the propagator $\mathcal{R}_t(\psi_0,u)$ of \eqref{eq:schro} at any time $t\in\mathbb{R}$, which is a solution of \eqref{eq:schro} in the weak sense \cite[Proposition 2.1\&Remark 2.7]{BMS}. 
Moreover, the quantum evolution is unitary, that is, for any $(\psi_0,u)\in L^2(M,\mathbb{C})\times L^1_{\rm loc}(\mathbb{R},\mathbb{R}^m)$ one has 
$$\|\mathcal{R}_t(\psi_0,u)\|_{L^2(M,\mathbb{C})}=\|\psi_0\|_{L^2(M,\mathbb{C})}, \forall t\in\mathbb{R}. $$
Let $\mathcal{S}=\{\psi\in L^2(M,\mathbb{C})\mid \|\psi\|_{L^2(M,\mathbb{C})}=1\}$ be the unit sphere of $L^2(M,\mathbb{C})$. 
\begin{definition}
We say that an element $\psi_1\in \mathcal{S}$ belongs to the \emph{small-time approximately reachable set} from $\psi_0\in \mathcal{S}$, and we write $\psi_1\in \overline{\rm Reach}_{\rm st}(\psi_0)$, if for every $\epsilon>0$ and $\tau>0$ there exist a time $T\in (0,\tau)$ and a control $u\in L^1([0,T],\mathbb{R}^m)$ such that 
$$\|\mathcal{R}_T(\psi_0,u)-\psi_1\|_{L^2(M,\mathbb{C})}<\epsilon.$$
\end{definition}
The characterization of small-time approximately reachable sets for Schr\"odinger partial differential equations is an open challenge. What is known is that for general initial data $\psi_0$ and on a general manifold, $\overline{\rm Reach}_{\rm st}(\psi_0)\neq \mathcal{S}$ \cite{minimal-time-approximate,obstruction-ivan}. Nevertheless, there are examples of conservative bilinear systems for which $\overline{\rm Reach}_{\rm st}(\psi_0)=\mathcal{S}$ for all $\psi_0\in \mathcal{S}$ \cite{minimal-time-thomas}. 

It is well-known that one can follow arbitrarily fast the directions spanned by the potentials of interaction $W_j$, $j=1,\dots,m$: this follows from the limit

$$\lim_{\delta \to 0}\exp\left(-i \delta\left(-\Delta_g+V+\sum_{j=1}^m\frac{u_j}{\delta}W_j\right)\right)\psi_0=\exp\left(-i\sum_{j=1}^m u_jW_j \right)\psi_0, $$
holding in $L^2(M,\mathbb{C})$ for any constant $u_j\in\mathbb{R}$ and $\psi_0\in L^2(M,\mathbb{C})$. This shows that for $\psi\in L^2(M,\mathbb{C})$
$$\left\{e^{i \phi}\psi_0\mid \phi\in{\rm span}\{W_1,\dots,W_m\}\right\}\subset  \overline{\rm Reach}_{\rm st}(\psi_0).$$

In \cite{duca-nersesyan}, Duca and Nersesyan showed that additional directions can be followed arbitrarily fast in \eqref{eq:schro}. They considered a $d$-dimensional torus, that is $M=\mathbb{T}^d:=\mathbb{R}^d/2\pi\mathbb{Z}^d$, with Cartesian coordinates $(x_1,\dots,x_d)$, and proved the following limit 
\begin{align}
\lim_{\delta \to 0}e^{-i\delta^{-1/2}\varphi}\mathcal{R}_\delta\left(e^{i\delta^{-1/2}\varphi}\psi_0,\sum_{j=1}^m\frac{u_j}{\delta}W_j\right)
=\exp\left(-i\sum_{i=1}^d\left(\frac{\partial \varphi}{\partial x_i}\right)^2-i \sum_{j=1}^mu_jW_j\right)\psi_0, \label{limit:alessandro}
\end{align}
holding in $H^s(\mathbb{T}^d,\mathbb{C})$, for any $\psi_0\in H^s(\mathbb{T}^d,\mathbb{C})$, $\varphi\in C^\infty(\mathbb{T}^d,\mathbb{R})$ and $u_j\in\mathbb{R}$ (here $s>s_d$, being $s_d$ the least integer strictly greater than $d/2$).
From \eqref{limit:alessandro}, they developed a saturation technique for multiplicative controls with trigonometric potential of interactions, and found that for $\psi_0\in H^s(\mathbb{T}^d,\mathbb{C})$
$$\{e^{i \phi}\psi_0\mid \phi\in C^\infty(\mathbb{T}^d,\mathbb{R})\}\subset \overline{\rm Reach}_{\rm st}(\psi_0).$$
We remark that this small-time controllability property in \cite{duca-nersesyan} is in fact proved for the harder problem of nonlinear Schr\"odinger equations. As a corollary of this result, they obtained the small-time approximate controllability among eigenstates: denoting by $\Phi:=\{\phi_k(x)=(2\pi)^{-d/2}\exp(i \langle k, x\rangle)$, $k\in \mathbb{Z}^d\}$ the set of eigenfunctions of the Laplacian of $\mathbb{T}^d$, they found that
$$\Phi\subset \overline{\rm Reach}_{\rm st}(\phi_k),\quad \forall k\in\mathbb{Z}^d. $$
Saturation techniques have been introduced by Agrachev and Sarychev \cite{navier-stokes, agrachev2} to study the approximate controllability of 2D Navier-Stokes and Euler systems with additive controls, and extended to the 3D case in \cite{Shirikyan1, Shirikyan2}. Other recent developments of these techniques are given, e.g., in \cite{coron-small-semiclassical} to study small-time controllability properties of semiclassical Schr\"odinger equations, and in \cite{duca-nersesyan2} to study local exact controllability of 1D Schr\"odinger equations with Dirichlet boundary conditions.
\subsection{Main results}
Here, we investigate properties similar to those studied in \cite{duca-nersesyan}, in the linear setting, starting from a general Riemannian manifold. Our first result is the following.
\begin{theorem}\label{thm:limit}
Let $M$ be a smooth closed manifold equipped with a Riemannian metric $g$. Let $V\in L^\infty(M,\mathbb{R})$, $W_j\in C^\infty(M,\mathbb{R}), j=1,\dots,m$. Then, for any $(u_1,\dots,u_m)\in\mathbb{R}^m$, $\psi_0\in L^2(M,\mathbb{C})$ and $\varphi\in C^\infty(M,\mathbb{R})$ the following limit holds in $L^2(M,\mathbb{C})$
\begin{align*}
\lim_{\delta \to 0}&e^{-i\delta^{-1/2}\varphi}\exp\left(-i\delta\left(-\Delta_g+V+\sum_{j=1}^m\frac{u_j}{\delta}W_j\right)\right)e^{i\delta^{-1/2}\varphi}\psi_0\\
=&\exp\left(-i g(\nabla_g \varphi,\nabla_g \varphi)-i \sum_{j=1}^mu_jW_j\right)\psi_0.
\end{align*}
\end{theorem}
Exactly as it is done in \cite{duca-nersesyan} in the case of $\mathbb{T}^d$, the limit given in Theorem \ref{thm:limit} can be applied in an iterative way to describe a small-time controllability property on $M$ (see Theorem \ref{thm:generalmanifold}). 

In the case of the two-dimensional sphere $S^2$ with trigonometric potential of interactions, we obtain the following result.
\begin{theorem}\label{thm:molecule}
Let $\psi_0\in L^2(S^2,\mathbb{C})$. Then, system \eqref{eq:S2} satisfies
$$\{e^{i \phi}\psi_0\mid \phi\in L^2(S^2,\mathbb{R})\}\subset \overline{\rm Reach}_{\rm st}(\psi_0).$$
\end{theorem}
As a corollary of Theorem \ref{thm:molecule} we obtain the small-time approximate controllability among particular eigenstates of the Laplace-Beltrami operator of $S^2$.
\begin{corollary}\label{cor:eigenfunctions}
Let $Y^j_m, j\in\mathbb{N},m=-j,\dots,j$, 
be the spherical harmonics, 
which are the eigenfunctions of $\Delta_{S^2}$. Then, system \eqref{eq:S2} satisfies
$$(-1)^jY^j_{\pm j}\in \overline{\rm Reach}_{\rm st}\left(Y^j_{\mp j}\right), \forall j\in\mathbb{N}.$$
\end{corollary}

\subsection{Structure of the paper}
In Section \ref{sec:proof-of-limit} we give a proof of Theorem \ref{thm:limit}, which is then applied in Section \ref{sec:saturation} to describe a small-time approximate controllability property for general manifolds. In Section \ref{sec:example} we develop this property on the 2-dimensional sphere, proving Theorem \ref{thm:molecule} and Corollary \ref{cor:eigenfunctions}. We conclude with an Appendix where we give an algebraic interpretation of Theorem \ref{thm:limit}.
\section{Proof of Theorem \ref{thm:limit}}\label{sec:proof-of-limit}
We start by defining for $\delta>0, t\in\mathbb{R}$
\begin{align*}
L_{\delta}&=e^{-i\delta^{-1/2}\varphi}\left(\!\!-\Delta_g+V+\sum_{j=1}^m\frac{u_j}{\delta}W_j\!\!\right)e^{i\delta^{-1/2}\varphi},\\
L&=g(\nabla_g \varphi,\nabla_g \varphi)+ \sum_{j=1}^mu_jW_j,
\end{align*}
as self-adjoint operators on $L^2(M,\mathbb{C})$ with common domain $H^2(M,\mathbb{C})$ (where $L$ is a multiplicative operator). We have the following.
\begin{lemma}\label{lem:limit}
Let $M$ be a smooth closed manifold equipped with a Riemannian metric $g$. Let $V\in L^\infty(M,\mathbb{R})$, $W_j\in C^\infty(M,\mathbb{R})$, $j=1,\dots,m$. Then, for any $(u_1,\dots,u_m)\in\mathbb{R}^m$, $\psi_0\in H^2(M,\mathbb{C})$ and $\varphi\in C^\infty(M,\mathbb{R})$ we have $\delta L_\delta\psi_0\to L\psi_0$ in $L^2(M,\mathbb{C})$ as $\delta \to 0$.
\end{lemma}
\begin{proof}
We compute
\begin{align*}
&e^{-i\delta^{-1/2}\varphi}\delta\Delta_ge^{i\delta^{-1/2}\varphi}\psi_0\\
=&i\delta^{1/2}(\Delta_g\varphi)\psi_0- g(\nabla_g \varphi,\nabla_g\varphi)\psi_0+\delta\Delta_g\psi_0+2i\delta^{1/2}g(\nabla_g \varphi,\nabla_g\psi_0),
\end{align*}
where we used that 
\begin{align*}
&\Delta_g(fh)=\Delta_gf+\Delta_gh+2g(\nabla_gf,\nabla_gh),\quad \Delta_gf=({\rm div}_{\omega_g}\circ\nabla_g)f,\\
& \nabla_g(e^f)=e^f\nabla_g(f), \quad {\rm div}_{\omega_g}(f \nabla_gh)=f \Delta_gh+g(\nabla_g h,\nabla_g f),
\end{align*}
for any functions $f,h\in H^2(M,\mathbb{C})$. 
The conclusion follows by letting $\delta\to 0$ thanks to the regularity of $\varphi,\psi_0$ and the compactness of $M$.
\end{proof}
The previous Lemma \ref{lem:limit} proves that the family of self-adjoint operators $\{\delta L_{\delta}\}_{\delta>0}$ with common domain $H^2(M,\mathbb{C})$ converges to $L$ strongly as $\delta\to 0.$
Hence, from \cite[Theorem VIII.25(a)]{rs1}, we also see that $\delta L_{\delta}\to L$ in the strong resolvent sense as $\delta\to 0.$ Applying Trotter's Theorem \cite[Theorem VIII.21]{rs1}, we conclude that $e^{-i\delta L_{\delta}}\psi_0\to e^{-iL}\psi_0$ in $L^2(M,\mathbb{C})$ as $\delta\to 0$ for any $\psi_0\in L^2(M,\mathbb{C})$. Let $\psi_0\in L^2(M,\mathbb{C})$, and define for $\delta>0$ and for any $t\in\mathbb{R}$,
\begin{align*}
\psi(t)&=e^{-itL_{\delta}}\psi_0,\\
\Psi(t)&=e^{i \delta^{-1/2}\varphi}\psi(t).
\end{align*}
Then, $\psi$ weakly solves
$$i \frac{d}{dt}\psi(t)=L_\delta\psi(t),\quad \psi(0)=\psi_0,$$
so, $\Psi$ weakly solves 
\begin{align*}
i \frac{d}{dt}\Psi(t)\!&=\!\left(\!\!-\Delta_g+V+\sum_{j=1}^m\frac{u_j}{\delta}W_j\!\!\right)\!\!\Psi(t),\\\Psi(0)&=e^{i\delta^{-1/2}\varphi}\psi_0.
\end{align*}
Then, necessarily 
\begin{align*}
\Psi(t)=\mathcal{R}_t\left(e^{i\delta^{-1/2}\varphi}\psi_0,\sum_{j=1}^m\frac{u_j}{\delta}\right)=\exp\!\!\left(\!\!-it\left(-\Delta+V+\sum_{j=1}^m\frac{u_j}{\delta}W_j\right)\!\!\right)\!\!e^{i\delta^{-1/2}\varphi}\psi_0,
\end{align*}
which implies
\begin{align*}
\psi(t)=e^{-i\delta^{-1/2}\varphi}\!\exp\!\!\left(\!\!-it\!\!\left(\!\!-\Delta_g+V+\sum_{j=1}^m\frac{u_j}{\delta}W_j\!\!\right)\!\!\right)e^{i\delta^{-1/2}\varphi}\psi_0,
\end{align*}
and concludes the proof of Theorem \ref{thm:limit}. 


\section{Small-time control in saturation spaces}\label{sec:saturation}
Following \cite{duca-nersesyan}, we associate with \eqref{eq:schro} a non-decreasing sequence of vector spaces. Let 
$$\mathcal{H}_1:={\rm span}_{\mathbb{R}}\{W_1,\dots,W_m\},$$
and for any $n\in\mathbb{N}, n>1$ define $\mathcal{H}_n$ as the largest real vector space whose elements can be written as
$$\varphi_0+\sum_{j=1}^N\alpha_jg(\nabla_g \varphi_j,\nabla_g \varphi_j), \quad\varphi_i \in \mathcal{H}_{n-1},\,\alpha_i\in\mathbb{R}\,\,\forall i=0,\dots,N,\quad N\in\mathbb{N}.$$
Consider the \emph{saturation} space $\mathcal{H}_\infty:=\bigcup_{n=1}^\infty \mathcal{H}_n$. We have the following.
\begin{theorem}\label{thm:generalmanifold}
Let $\psi_0\in L^2(M,\mathbb{C})$. Then, system \eqref{eq:schro} satisfies
$$\{e^{i\phi}\psi_0\mid \phi\in\mathcal{H}_\infty\}\subset \overline{\rm Reach}_{\rm st}(\psi_0).$$
\end{theorem}
The proof of Theorem \ref{thm:generalmanifold} is analogous to the proof of Theorem 2.2 in \cite{duca-nersesyan}. We sketch it here for completeness.
\begin{proof}[Sketch of the proof of Theorem \ref{thm:generalmanifold}]
It suffices to prove by induction on $n\in\mathbb{N}$ that for any $\psi\in L^2(M,\mathbb{C})$ one has
\begin{equation}\label{eq:saturation}
\left\{e^{i \phi}\psi_0\mid \phi\in\mathcal{H}_n\right\}\subset  \overline{\rm Reach}_{\rm st}(\psi_0).
\end{equation}
One does it by iteratively applying the limit of conjugated trajectories given in Theorem \ref{thm:limit}.
As basis of induction we compute the limit of Theorem \ref{thm:limit} with $\varphi=0$: this proves that a control law $(u_1,\dots,u_m)/\delta$ steers the system \eqref{eq:schro} from $\psi_0$ arbitrarily close to $\exp\left(-i \sum_{j=1}^mu_jW_j\right)\psi_0$ if the time $t=\delta$ is small enough; this means that for any $\psi_0\in L^2(M,\mathbb{C})$
$$\left\{e^{i \phi}\psi_0\mid \phi\in\mathcal{H}_1\right\}\subset  \overline{\rm Reach}_{\rm st}(\psi_0).$$
The idea is that we can now apply the limit of Theorem \ref{thm:limit} with $\varphi\in\mathcal{H}_1$: the limit is a composition of three exponentials that approximates a trajectory of \eqref{eq:schro} and at the same time approximates the state $\exp\left(-ig(\nabla_g\varphi,\nabla_g\varphi)\right)\psi_0$ if the time $t=\delta$ is small enough, where now $g(\nabla_g\varphi,\nabla_g\varphi)$ belongs to the larger vector space of directions $\mathcal{H}_2$. Notice that we are allowed to iterate this procedure because the potentials $W_j$ are smooth.

More precisely, assume that \eqref{eq:saturation} holds for $n'<n$ and let $\varphi_0+\sum_{j=1}^N\alpha_j g(\nabla_g \varphi_j,\nabla_g \varphi_j)\in\mathcal{H}_n$ where $\varphi_i\in\mathcal{H}_{n-1}$ for all $i=0,\dots,N$ and $\alpha_j\in\mathbb{R}$. If $\alpha_1\geq 0$, consider the limit of Theorem \ref{thm:limit} with $\varphi=-\alpha^{1/2}\varphi_1$, $u=0$, and initial condition $\exp(i\varphi_0)\psi_0$ (notice that it is possible to consider such an initial condition because of the inductive hypothesis). The application of the limit, together with the fact that by inductive hypothesis there exists a trajectory of \eqref{eq:schro} arbitrarily close (as the time $t=\delta$ gets smaller) to the composition of the three exponentials given in the limit, one has that for any 
$\psi_0\in L^2(M,\mathbb{C})$
\begin{equation}\label{eq:inductivestep}
\exp\left(i\varphi_0+i|\alpha_1|g(\nabla_g \varphi_1,\nabla_g \varphi_1)\right)\psi_0\in  \overline{\rm Reach}_{\rm st}(\psi_0).
\end{equation}
If $\alpha_1<0$, one only needs to replace $\delta$ with $\widetilde{\delta}=-\delta$ in the limit of Theorem \ref{thm:limit}, obtaining $-|\alpha_1|$ instead of $|\alpha_1|$ in \eqref{eq:inductivestep}. By iterating this argument (that is, by considering the limit of Theorem \ref{thm:limit} with initial condition the LHS of \eqref{eq:inductivestep}, $\varphi=-\alpha_2^{1/2}\varphi_2$, and $u=0$ and so on) one obtains 
$$\exp\left(i\varphi_0+i\sum_{j=1}^N\alpha_jg(\nabla_g \varphi_j,\nabla_g \varphi_j)\right)\psi_0\in  \overline{\rm Reach}_{\rm st}(\psi_0). $$
\end{proof}
\section{Example: the 2-dimensional sphere}\label{sec:example}
In this section we show how to obtain Theorem \ref{thm:molecule} and Corollary \ref{cor:eigenfunctions}. In particular, we prove that the saturation space $\mathcal{H}_\infty$ associated with the potentials of interaction
$$W_1=x,\quad W_2=y,\quad W_3=z, $$
seen as polynomials on $S^2=\{(x,y,z)\in\mathbb{R}^3, x^2+y^2+z^2=1\}$, is dense in $L^2(S^2,\mathbb{R})$. For any $n\in\mathbb{N}$, let $\mathbb{P}_n$ be the vector space of real polynomials $p:S^2\subset\mathbb{R}^3\to \mathbb{R}$ of degree less or equal than $n$. 
We have the following.
\begin{lemma}\label{lem:saturation}
For any $n\geq 2$, $\mathbb{P}_n\subset\mathcal{H}_n.$
\end{lemma}
By density of polynomials in $L^2(S^2,\mathbb{R})$, Theorem \ref{thm:generalmanifold} and Lemma \ref{lem:saturation} imply Theorem \ref{thm:molecule}. Moreover, by noticing that
$$Y^j_{\pm j}(\alpha,\beta)=\frac{(\mp 1)^j}{2^j j!}\sqrt{\frac{(2j+1)!}{4\pi}}\sin^j(\beta)e^{\pm i j\alpha},$$
Corollary \ref{cor:eigenfunctions} is then a straightforward consequence of Theorem \ref{thm:molecule}: it suffices to approximate the (discontinuous) functions $\pm 2j \alpha$ in $L^2(S^2,\mathbb{R})$ with polynomials. We are thus left to prove Lemma \ref{lem:saturation}.
\begin{proof}[Proof of Lemma \ref{lem:saturation}]
We prove the statement by induction. By definition, we have
$$\mathcal{H}_1:={\rm span}\{x,y,z\}.$$
To prove the basis of induction, it suffices to prove that the monomials 
$$1,x^2,y^2,z^2,xy,xz,yz$$
belong to $\mathcal{H}_2$. We recall that, since the Riemannian metric $g$ on $S^2$ is the pull-back metric induced by the inclusion $S^2 \hookrightarrow\mathbb{R}^3$, for any smooth function on the sphere $\varphi=\varphi(x,y,z)$ the Riemannian gradient $\nabla_g\varphi$ is equal to the vector field $\nabla_{S^2}\varphi$ in $\mathbb{R}^3$ tangent to the sphere given by
$$\nabla_{S^2}\varphi=\sum_{i=1}^3 \left(\nabla_{S^2}\varphi\right)_i \frac{\partial}{\partial x_i} ,\quad  \left(\nabla_{S^2}\varphi\right)_i=\sum_{j=1}^3\frac{\partial \varphi}{\partial x_j}(\delta_{ij}-x_ix_j),$$
where $x_1=x,x_2=y,x_3=z$, and the Riemannian norm of $\nabla_{S^2}\varphi$ can thus be computed as a scalar product in $\mathbb{R}^3$, i.e.,
$$g(\nabla_{g}\varphi,\nabla_{g}\varphi)=\langle\nabla_{S^2}\varphi,\nabla_{S^2}\varphi \rangle=\sum_{i=1}^3 \left(\nabla_{S^2}\varphi\right)_i^2.$$
 Hence, we compute for $i,j\in\{1,2,3\}$
$$
\left(\nabla_{S^2}x_j\right)_i=\begin{cases}
-x_ix_j,&i\neq j\\
1-x_j^2,&i=j
\end{cases}
$$
For $j=3$ we get
$$\langle\nabla_{S^2}z ,\nabla_{S^2}z \rangle=z^2(x^2+y^2)+(1-z^2)^2=z^2(1-z^2)+(1-z^2)^2=1-z^2,$$
where we used that $x^2+y^2=1-z^2$ on the sphere. Analogously we obtain
$$ \langle\nabla_{S^2}x_j ,\nabla_{S^2}x_j \rangle=1-x_j^2,\quad j=1,2,3. $$
We take the sum and use that $x^2+y^2+z^2=1$ on the sphere, obtaining
$$ \sum_{j=1}^3 \langle\nabla_{S^2}x_j ,\nabla_{S^2}x_j \rangle=2,$$
from which we see that $1,x^2,y^2,z^2$ are in $\mathcal{H}_2$. Then, we also compute
$$
\left(\nabla_{S^2}(x\pm z)\right)_i=\begin{cases}
1-x^2\mp xz,&i=1\\
-xy\mp yz,&i=2\\
-xz\pm (1-z^2),&i=3
\end{cases}
$$
from which we get
\begin{align*}
&\langle\nabla_{S^2}(x-z) ,\nabla_{S^2}(x-z) \rangle-\langle\nabla_{S^2}(x+z) ,\nabla_{S^2}(x+z) \rangle\\
=& (1-x^2+xz)^2+(-xy+yz)^2+(-xz-1+z^2)^2\\
-&(1-x^2-xz)^2-(-xy-yz)^2-(-xz+1-z^2)^2\\
=&8xz-4x^3z-4xy^2z-4xz^3=4xz,
\end{align*}
where we used that $y^2=1-x^2-z^2$ on the sphere. This implies that $xz\in\mathcal{H}_2$. Since everything is symmetric in $(x,y,z)$, the same argument can of course be repeated with $y$ instead of $z$, obtaining that $xy\in \mathcal{H}_2$, and $y$ instead of $x$, obtaining that $yz\in \mathcal{H}_2$. This proves the basis of induction.

We now show that if the statement holds for all $n'<n$, then it holds for $n$. Notice that thanks to the inductive hypothesis, in order to prove the statement it suffices to show that the monomials 
$$x^ky^lz^m,\quad  (k,l,m)\in\mathbb{N}^3,\quad k+l+m=n,$$
are in $\mathcal{H}_n$. We thus compute for $k,m\neq 0$
$$
\left(\nabla_{S^2}(z^k\pm z^m)\right)_i=\begin{cases}
-kz^kx\mp mz^mx,&i=1\\
-kz^ky\mp mz^my,&i=2\\
kz^{k-1}(1-z^2)\pm mz^{m-1}(1-z^2),&i=3
\end{cases}
$$
which gives
$$\langle\nabla_{S^2}(z^k- z^m) ,\nabla_{S^2}(z^k- z^m) \rangle-\langle\nabla_{S^2}(z^k+ z^m) ,\nabla_{S^2}(z^k+ z^m) \rangle=4kmz^{k+m}-4kmz^{k+m-2}.$$
By choosing $k+m=n$, since $z^k,z^m,z^{k+m-2}\in \mathcal{H}_{n-1}$ by inductive hypothesis, we obtain that $z^{n}\in \mathcal{H}_n$. The same argument can of course be repeated with $x$ or $y$ instead of $z$, obtaining that $x^{n},y^n\in \mathcal{H}_n$. We then compute
$$
\left(\nabla_{S^2}(x^k\pm y^lz^m)\right)_i=\begin{cases}
kx^{k-1}(1-x^2)\mp mxy^lz^m\mp lxy^lz^m,&i=1\\
-kx^ky\mp my^{l+1}z^m\pm ly^{l-1}(1-y^2)z^m,&i=2\\
-kx^kz\pm my^lz^{m-1}(1-z^2)\mp ly^lz^{m+1},&i=3
\end{cases}
$$
which gives
$$\langle\nabla_{S^2}(x^k- y^lz^m) ,\nabla_{S^2}(x^k- y^lz^m) \rangle-\langle\nabla_{S^2}(x^k+ y^lz^m) ,\nabla_{S^2}(x^k+ y^lz^m) \rangle=4k(l+m)x^ky^lz^{m}.$$
By choosing $m=0,k\neq 0, k+l=n$, or $l=0,k\neq 0, k+m=n$, since $x^k,y^lz^m\in\mathcal{H}_{n-1}$ by inductive hypothesis, we obtain that $x^ky^l\in \mathcal{H}_n$ or that $x^kz^m\in \mathcal{H}_n$. By exchanging the roles of $x$ and $y$, the same argument can of course be repeated, obtaining that $y^kz^m\in \mathcal{H}_n$. Finally, by choosing $k,m,n\neq 0,k+m+l=n$, we obtain that $x^ky^lz^m\in\mathcal{H}_n$, which concludes the proof.
\end{proof}
We conclude this section by noticing that, since 
$$-\Delta_{S^2}Y^j_m=j(j+1)Y^j_m,\quad\forall j\in\mathbb{N},\,\, m=-j,\dots,j, $$
the small-time transfer between $(-1)^jY^j_{\pm j}$ and $Y^j_{\mp j}$ obtained in Corollary \ref{cor:eigenfunctions} happens between two eigenfunctions that correspond to the same degenerate eigenvalue $j(j+1)$.
\section{Conclusion}
We proved that it is in principle possible to obtain a transfer of population in arbitrarily small times among particular eigenstates of the physically relevant system of a rotating rigid molecule. Extensions of small-time controllability among more general states in such systems is an open challenge. \\
The modelling of controlled quantum systems via perturbation of a stationary Schr\"odinger equation is valid as long as the external field varies sufficiently slowly and its amplitude is small enough. The results of this paper should thus be interpreted as the fact that, for these particular eigenstates transfers, there is no theoretical lower bound on the time. The actual limitation on the minimal time needed to obtain the transfer is then due to the validity of the model w.r.t. the size of the control.
\section*{Acknowledgments}
      The authors thank Alessandro Duca for several discussions and Herschel Rabitz for useful remarks.\\
This work is part of the project CONSTAT, supported by the Conseil 
Régional de Bourgogne Franche-Comté and the European Union through the PO FEDER 
Bourgogne 2014/2020 programs, by the French ANR through the grant QUACO (ANR-17-CE40-0007-01) and by EIPHI Graduate School (ANR-17-EURE-0002).   

\appendix
\section{An heuristic in terms of Lie brackets}
Here we interpret Theorem \ref{thm:limit} in an algebraic way. Let us rewrite \eqref{eq:schro} in abstract terms as
\begin{equation}\label{eq:schro-ab}
i \frac{d}{dt}\psi(t)=\left(H_0+\sum_{j=1}^m u_j(t)H_j\right)\psi(t),\,\,\psi(0)=\psi_0
\end{equation}
where $\psi$ belongs to some infinite-dimensional Hilbert space $\mathcal{H}$.
\begin{theorem}\label{thm:abstract}
Let $H_0$ be an unbounded self-adjoint operator with domain $\mathcal{D}(H_0)$, and $H_1,\dots,H_m$ be bounded self-adjoint operators. Let $S$ be a bounded self-adjoint operator satisfying 
\begin{align}
&[S,H_j]=0, \quad j=1,\dots,m,\label{rel:a}\\
&S\mathcal{D}(H_0)\subset \mathcal{D}(H_0),\quad {\rm ad}^3_S(H_0)\mathcal{D}(H_0)=0.\label{rel:b}
\end{align}
Then, for any $\psi_0\in\mathcal{H}$ the following limit holds in $\mathcal{H}$
\begin{align}\label{limit:abstract}
\lim_{\delta \to 0}&e^{-i\delta^{-1/2}S}\exp\left(-i\delta\left(H_0+\sum_{j=1}^m\frac{u_j}{\delta}H_j\right)\right)e^{i\delta^{-1/2}S}\psi_0\nonumber \\
=&\exp\left(\frac{i}{2} {\rm ad}^2_S(H_0)-i\sum_{j=1}^m u_jH_j\right)\psi_0,
\end{align}
where ${\rm ad}^0_AB=B$, ${\rm ad}_AB=[A,B]=AB-BA$ and ${\rm ad}^{n+1}_AB=[A,{\rm ad}^{n}_AB]$. 
\end{theorem}
In the case of a quantum particle on a Riemannian manifold (see Theorem \ref{thm:limit}), where $H_0=-\Delta_g+V$ and $S=\varphi\in C^\infty(M,\mathbb{R})$, we have that ${\rm ad}^3_S(H_0)\psi_0=0$ for any $\psi_0\in H^2(M,\mathbb{C})$ and
$$ \frac{1}{2}{\rm ad}^2_S(H_0)\psi_0=-g(\nabla_g \varphi,\nabla_g \varphi)\psi_0,\quad \forall \psi_0\in H^2(M,\mathbb{C}).$$
When the Hilbert space $\mathcal{H}$ is finite-dimensional, it is easy to check that the bracket relation ${\rm ad}^3_S(H_0)=0$ implies ${\rm ad}^2_S(H_0)=0$ and $[S,H_0]=0$ on $\mathcal{D}(H_0)$ (so the limit \eqref{limit:abstract} does not furnish any additional direction). Interestingly, as we have just noticed, this is not the case when $\mathcal{H}$ is infinite-dimensional. This is related to the fact that $\varphi\in C^\infty(M,\mathbb{R})$ (seen as a multiplication operator in $L^2(M,\mathbb{C})$) has continuous spectrum (if it is not a constant). 
\begin{proof}[Proof of Theorem \ref{thm:abstract}.]
In order to prove \eqref{limit:abstract}, it suffices to prove the analogous of the limit given in Lemma \ref{lem:limit}, i.e.,
\begin{align}\label{limit:abstract2}
\lim_{\delta \to 0}e^{-i\delta^{-1/2}S}\delta\left(H_0+\sum_{j=1}^m\frac{u_j}{\delta}H_j\right)e^{i\delta^{-1/2}S}\psi_0
=\left(-\frac{1}{2} {\rm ad}^2_S(H_0)+ \sum_{j=1}^mu_jH_j\right)\psi_0,
\end{align}
for any $\psi_0\in\mathcal{D}(H_0)$, and then repeat the same steps as we did in Section \ref{sec:proof-of-limit} (using the self-adjointness of $H_0,H_j$ and $S$). Since $S$ is bounded, we can use the Baker-Campbell-Hausdorff formula and write
 \begin{align*}
& e^{-i\delta^{-1/2}S}\delta\left(H_0+\sum_{j=1}^m\frac{u_j}{\delta}H_j\right)e^{i\delta^{-1/2}S}\psi_0\\
 =&\sum_{k=0}^\infty\frac{1}{k!}{\rm ad}^k_{(-i\delta^{-1/2}S)}\left(\delta\left(H_0+\sum_{j=1}^m\frac{u_j}{\delta}H_j\right)\right)\psi_0\\
 =&\sum_{k=0}^2\frac{(-i)^{k}\delta^{-k/2+1}}{k!}{\rm ad}^k_S(H_0)\psi_0+\sum_{j=1}^mu_jH_j\psi_0
 \end{align*}
where we used the commutator relations \eqref{rel:a} and \eqref{rel:b} (and the fact that \eqref{rel:b} implies \newline${\rm ad}^k_S(H_0)\mathcal{D}(H_0)=0$ for all $k\geq 3$) in the second equality. The proof of \eqref{limit:abstract2} is concluded, and the proof of Theorem \ref{thm:abstract} follows.
\end{proof}

\bibliographystyle{spmpsci}
\bibliography{references}

 \end{document}